\documentclass[11pt]{article}
\usepackage{amsmath}
\usepackage{physics}
\usepackage{amscd}
\usepackage{amsthm}
\usepackage{amssymb} \usepackage{latexsym}
\usepackage{eufrak}
\usepackage{euscript}
\usepackage{bbm}
\usepackage[width=19cm, height=24cm]{geometry}
\usepackage{epsfig}
\usepackage{tikz}
\usepackage{graphics}
\usepackage{array}
\usepackage{enumerate}
\usepackage{authblk}
\usepackage{hyperref}
\usepackage[utf8]{inputenc}
\theoremstyle{theorem}
\newtheorem{theorem}{Theorem}
\theoremstyle{corollary}
\newtheorem{corollary}{Corollary}
\theoremstyle{lemma}
\newtheorem{lemma}{Lemma}
\theoremstyle{definition}
\newtheorem{definition}{Definition}
\theoremstyle{proposition}
\newtheorem{proposition}{Proposition}
\theoremstyle{proof}
\newtheorem{example}{Example}

\theoremstyle{remark}

\newcommand{\bel}[1]{\begin{equation}\label{#1}}

\newcommand{\be}{\begin{equation}}

\newcommand{\ba}{\begin{eqnarray}}
\newcommand{\ea}{\end{eqnarray}}

\newcommand{\qe}{\end{equation}}

\newcommand{\bd}{\begin{definition}}
\newcommand{\ed}{\end{definition}}
\newcommand{\beg}{\begin{example}}
\newcommand{\eeg}{\end{example}}

\newcommand{\tuple}[2]{{#1}_1,{#1}_2,\dots,{#1}_{#2}}
\newcommand{\tensum}[2] { \sum_{i_{2}, i_{3}, \dots,  i_{m}=1}^{n} #1_{i i_{2}i_{3}\dots i_{m}} #2_{i_{2}}x_{i_{3}} \dots x_{i_{m}}.}
\begin{document}
\title{On the spectrum of directed uniform and non-uniform hypergraphs}

\author[]{\rm Anirban Banerjee}
\author[]{\rm Arnab Char}

\affil[]{Department of Mathematics and Statistics}
\affil[ ]{Indian Institute of Science Education and Research Kolkata}
\affil[ ]{Mohanpur-741246,  India}
\affil[ ]{\textit {{\scriptsize anirban.banerjee@iiserkol.ac.in, arnabchar@gmail.com}}}

\maketitle 
\begin{abstract}
Here, we suggest a method to represent general directed uniform and non-uniform  hypergraphs by different connectivity tensors.  We  show many results on spectral properties of undirected hypergraphs  also hold for general directed uniform  hypergraphs.
Our representation of  a connectivity tensor will be very useful for the further development in spectral theory of directed hypergraphs. At the end, we have also introduced the concept of weak* irreducible hypermatrix to better explain connectivity of a directed hypergraph.
\end{abstract}


\section{Introduction}
In 2005, Qi introduced 
the concept of eigenvalues of a real supersymmetric tensor  \cite{Qi2005}. It stimulated many researchers to work rigorously  on the different spectral properties of tensors \cite{Chang2009, Li2013, Qi2013, Shao2013}. Perrone-Frobenius theorem is also introduced for tensors \cite{Chang2008}. 
In 2012, Cooper and Dutle \cite{Cooper2012} 
defined adjacency tensor for uniform hypergraphs and studied its eigenvalues. Afterwards, many work has been started on  spectral properties of different tensors (or hypermatrices) which represent hypergraphs \cite{Hu2013_2, Hu2014, Hu2015, Pearson2015, Qi2014, Qi2014_2}. 
The most of the studies were confined to uniform hypergraphs. In 2016, Banerjee and others represented non-uniform hypergraphs by tensors and studied their spectrum  \cite{Banerjee2017}.
We also refer \cite{Qi2017}, which has accumulated many information on spectral analysis of hypergraphs using different tensors, for more reading on the recent developments in this area. 
In 2016, Xie and Qi represented directed uniform hypergraphs by different connectivity tensors and explored properties of their various eigenvalues \cite{Xie2016}. This spectral study only focuses on the very specific kind of directed uniform hypergraphs, where only one vertex is always in the tail of an directed (hyper) edge.


Here, we show a mathematical framework to represent a general directed uniform hypergraph by different connectivity tensors. The tails (or heads) of (directed) edges in a  general directed ($m$-)uniform hypergraph may contain different (non zero) number of vertices, but, the total number of vertices in the tail and head of an edge is constant ($m$). We have studied different spectral properties of  adjacency tensors, Laplacian tensors and signless  Laplacian tensors of general directed uniform hypergraphs. This spectral study has also been extended for general directed non-uniform hypergraphs, where the total number of vertices in the tail and head of a directed  edge is not constant.

\section{Preliminary}
An  $m$ order $n$  dimensional real hypermatrix is a multidimensional array and is defined as, $$A=(a_{i_{1}{i_{2}}\dots{i_{m}}}),   a_{i_{1}{i_{2}}\dots{i_{m}}} \in \mathbb{R} \text{ and } 1\leq i_{1}, i_{2}, \dots, i_{m} \leq n.$$
Now onwards, we use hypermatrix and tensor interchangeably. Let us denote the set of all $m$-order $n$-dimensional hypermatrices (tensors) by $M_{m,n}$.\\
 A real hypermatrix $A=(a_{i_{1}{i_{2}}\dots{i_{m}}})\in M_{m,n}$ is called supersymmetric if its entries, $a_{i_{1}{i_{2}}\dots{i_{m}}}$'s, are invariant under any permutation of the indices. A tensor is called non-negative if all of its entries are non-negative real numbers.\\
 Now we recall the hypermatrix (tensor) product defined in \cite{Shao2013}. Let $A\in M_{m,n}$ and $B\in M_{k,n}$. The product of $A$ and $B$ is another hypermatrix $C=(c_{i\alpha_1\alpha_2\dots \alpha_{m-1}})=AB\in M_{(m-1)(k-1)+1,n}$ and it is defined as,
$$c_{i\alpha_1\alpha_2\dots \alpha_{m-1}}=\sum_{i_2\dots i_m=1}^n a_{ii_2\dots i_m}b_{i_2\alpha_1}\dots b_{i_m\alpha_{m-1}},$$ $\text{  where } i\in\{1,2,\dots ,n\} \text { and } \alpha_1,\alpha_2,\dots ,\alpha_{m-1}\in \{1,2,\dots ,n\}^{k-1}, \text{ that is, } k-1 \text{ times cartesian product of } \{1,2,\dots ,n\} $\\
By the above definition, if $x=(x_1,x_2,\dots,x_n)\in M_{1,n}$ and $A\in M_{m,n}$ then,
$Ax$  becomes an $n$-dimensional vector, whose $i$-th entry is $$(Ax)_i=\tensum{A}{x}$$
We also denote $Ax$ by $Ax^{m-1}.$ For $A\in M_{m,n}$ and $x\in M_{1,n}$, we define $$Ax^m:=\sum_{i_{1}, i_{2}, \dots,  i_{m}=1}^{n} a_{i_1 i_{2}\dots i_{m}} x_{i_{1}}x_{i_{2}} \dots x_{i_{m}},$$ which is a homogeneous polynomial of $x_1,x_2,\dots ,x_m$. For any $x\in M_{1,n}$ we also define a vector $x^{[m-1]}$ whose $i$-th entry is given by, $$(x^{[m-1]})_{i}=x_i^{m-1}.$$

\begin{definition}
  Let $A\in M_{m,n}$ be a nonzero hypermatrix. A pair $(\lambda,  x) \in \mathbb{C} \times (\mathbb{C}^{n}\setminus \{0\})$  is called eigenvalue and  eigenvector (or simply an eigenpair) if they satisfy the following equation,
  $$Ax^{m-1}=\lambda x^{[m-1]},$$
  that is, an eigenpair $(\lambda,x)$ satisfies the following $n$ equations,
  $$(Ax)_i=\lambda x_i^{m-1},\text{ } i=1,2,\dots ,n.$$
  For any $x=(x_i)\in M_{1,n}$, we denote $x\in \mathbb{R}_{++}^{n}$ or $x\in \mathbb{R}_{+}^{n}$ if  all the entries, $x_i$'s are positive or non-negative respectively. We denote $$\rho(A)=\{|\lambda|:\lambda \text{ is an eigenvalue of } A\}.$$
   $(\lambda,  x)$, is called an $H$-eigenpair (i.e.,  $\lambda$ and $x$ are called $H$-eigenvalue and   $H$-eigenvector,  respectively) if they both are real. An $H$-eigenvalue $\lambda$ is called $H^+  (H^{++})$-eigenvalue  if the corresponding eigenvector 
   $x \in \mathbb{R}_+^n $
   $(\mathbb{R}_{++}^n)$.\\
\end{definition}

 \begin{definition}
  Let $A$ be a nonzero hypermatrix. A pair $(\lambda,  x) \in \mathbb{C} \times (\mathbb{C}^{n}\setminus \{0\})$  is called an $E$-eigenpair (where $\lambda$ and $x$ are called $E$-eigenvalue and  $E$-eigenvector,  respectively)
  if they satisfy the following equations,
  $$Ax^{m-1}= \lambda x,$$
  $$ \sum_{i=1}^{n} x_{i}^2 =1.$$
   $(\lambda^Z,  x)\in (\mathbb{R}\times (\mathbb{R}^n\setminus \{0\})$ is  called a $Z$-eigenpair if they satisfy the above equations.\\
 \end{definition}

\begin{definition}
 $A\in M_{m,n}$ is called copositive if for any $x\in {\mathbb{R}^n}_{+}$, $Ax^m\geq 0$.\\
\ed
By \cite{Qi2005}, for any non-symmetric tensor $A\in M_{m,n}$, we have a super-symmetric tensor $\overline{A}\in M_{m,n}$ such that $Ax^m=\overline{A}x^m$. Clearly $A$ is copositive if $\overline{A}$ is copositive. Now we have the following lemma.\\
\begin{lemma}{\label{cop}}
If $A\in M_{m,n}$ is copositive and it has an $H^+$ eigenvalue $\lambda$, then $\lambda\geq 0$.\\ 
\end{lemma}
\begin{proof}
Suppose, $\lambda$ is an eigenvalue and it's corresponding eigenvector is $x$ then,
$Ax^{m-1}=\lambda x^{[m-1]}$. This imply $\lambda=\dfrac{Ax^m}{{\abs{\abs{x}}_m}^m}$. Since $A$ is a copositive tensor, we have, $\lambda\geq 0$.\\
\end{proof}
\begin{theorem}[Lemma 3.1, \cite{Li2013}]\label{maxz}
Let $A\in M_{m,n}$ be a real symmetric tensor where $m$ is even. Then
$$\lambda_{max}^Z(A)= max _{||x||=1} x^{t}Ax^{m-1}=max _{||x||=1}Ax^m.$$\\
\end{theorem}
Let $A\in M_{m,n}$ and $r_i(A)=\sum\limits_{i_{2}, i_{3}, \dots,  i_{m}=1}^{n} a_{ii_{2}i_{3}\dots i_{m}}$, then we have the following theorem.\\
\begin{lemma}[Lemma 4, \cite{Yuan2015}]\label{bound}
Let $A\in M_{m,n}$ be a non-negative tensor. Then, $$\min_{i}\{r_i(A)\}\leq \rho(A)\leq \max_i\{r_i(A)\}$$
\end{lemma}
The Theorem 6(a)in \cite{Qi2005} can also be stated for any non-symmetric tensor $A\in M_{m,n}$ as  follows. 
\begin{theorem}\label{disk}
All the eigenvalues of $A$ will lie in the union of $n$ disks in $\mathbb{C}$. These $n$ disks have the diagonal elements of $A$ as their centres, and their corresponding radii are the absolute value of the sum of the corresponding off diagonal elements.\\  
\end{theorem}
\begin{definition} 
For a real hypermatrix $A = a_{i_1i_2\dots i_m}\in M_{m,n}$, we say that A is reducible if there
exists a nonempty proper index subset $J\subset[n]$ such that
$$a_{i_1\dots i_m }= 0,\text{ } \forall i_1\in J, \text{ }i_2,\dots i_m \notin J.$$
\end{definition}
A hypermatrix is irreducible if it is not reducible. In linear algebra it is observed that, a matrix is irreducible if and only if the underlying directed graph is strongly connected, but in case of hypermatrix theory, to explain connectivity the concept of weakly irreducible tensors has been introduced.\\
\begin{definition} 
Let $A = a_{i_1i_2\dots i_m}\in M_{m,n}$ be a real hypermatrix. Construct a directed graph $G = (V,E)$,
with the vertex set $V = [n]$, and directed edge $\vec{ij}\in E$, for all $j\in i_2,\dots i_m$ such
that $a_{i_1i_2\dots i_m}$. We call A is weakly irreducible if the associated graph G is strongly
connected.
\end{definition}

\section{Introduction to directed hypergraphs}

A hypergraph $G$ is a pair $(V, E)$ where $V$ is a set of elements called vertices,  and $E$ is a set of non-empty subsets of $V$ called  edges. Therefore,  E is a subset of $\mathcal{P}(V) \setminus\{\emptyset\}$,  where $\mathcal{P}(V)$ is the power set of $V$.

\bd{(Directed Hyperedge)}
\\Let $H=(V,E)$ be a hypergraph. An edge  $e\in E$ is called directed hyperedge if it has a partition $T_e$ (Tail) and $H_e$ (head), such that, the direction of the edge is given from tail to head.\\
\noindent A directed hyperedge $e$ is of length $m$, if $\abs{e}=m$. $\abs{T_e}$ and $\abs{H_e}$ are called tail length and head length, respectively, of the edge $e$.\\  
\ed
\bd{(Directed Hypergraph)}
\\A hypergraph $H=(V,E)$ is called directed hypergraph if it satisfies the following properties
\begin{enumerate}
\item each edge is directed edge and
\item for any two edges $e_1,e_2\in E$, $T_{e_1}\cup H_{e_1}\neq T_{e_2}\cup H_{e_2}$\\
\end{enumerate}
\ed
\noindent A directed hypergraph $H=(V,E)$ is called $m$-uniform, $\abs{e}=m,\text{ } \forall e\in E$. A directed hypergraph which is not uniform, is called directed non-uniform hypergraph. In a directed hypergraph we have two different notions for vertex degree.
\bd(Degree of a vertex)
\\Let $v\in V$ be a vertex of a directed hypergraph $H=(V,E)$. The out-degree  of $v$ in $H$ is given by,
$$d_v^{+}=\abs{\{e\in E:v\in T_{e}\}}.$$
and the in-degree is defined as,
$$d_v^{-}=\abs{\{e\in E:v\in H_{e}\}}.$$  \\
\ed
\noindent We denote $\Delta^+$, $\Delta^-$,$\delta^+$,$\delta^-$ as the highest out-degree, highest in-degree, smallest out-degree and smallest in-degree, respectively.
\section{Adjacency hypermatrix for a directed uniform hypergraph}
\bd
  Let $H=(V, E)$ be a directed $m$-uniform hypergraph, where, $V= \{ v_{1},  v_{2},  \dots,  v_{n} \}$ and $E = \{ e_{1},  e_{2},  \dots,  e_{k} \}$.
We define\textit{ the out-adjacency hypermatrix} , 
$${{A}_H}^{+} = (a_{i_{1}i_{2}\dots i_{m}}^+)\in M_{m,n}\text{ of } H \text{, as follows.}$$
For all edges $e=\{ v_{l_{1}},  v_{l_{2}},  \dots, v_{l_{k}},v_{l_{k+1}},\dots ,v_{l_m}\}\in E$, such 
that,  $T_e=\{ v_{l_{1}},  v_{l_{2}},  \dots, v_{l_{k}}\}$ and $H_e=\{v_{l_{k+1}},\dots ,v_{l_m}\}$,$$a_{i_{1}i_{2}\dots i_{k} i_{k+1\dots i_{m}}}^+=\dfrac{1}{(m-k)!(k-1)!}\text{,}$$where $i_1,i_2,\dots ,i_k$  are all distinct elements of $\tuple{l}{k}$ and $i_{k+1},\dots ,i_m$ are all distinct elements of $l_{k+1},\dots ,l_m$ and the rest of the entries are zero. \\
Similarly we define the \textit{in-adjacency hypermatrix},
  $${{A}_H}^{-} = (a_{i_{1}i_{2}\dots i_{m}}^-)\in M_{m,n},\text{ of } H.$$

For all edges $e=\{ v_{l_{1}},  v_{l_{2}},  \dots, v_{l_{m-k}},v_{l_{m-k+1}},\dots ,v_{l_m}\}\in E$, with  $T_e=\{ v_{l_{1}},  v_{l_{2}},  \dots, v_{l_{k}}\}$ and $H_e=\{v_{l_{k+1}},\dots ,v_{l_m}\}$ $$a_{i_{1}i_{2}\dots i_{k} i_{k+1}\dots i_{m}}^-=\dfrac{1}{(m-k-1)!k!}\text{,}$$where $i_1,i_2,\dots ,i_{m-k}$  are all distinct elements of $l_{k+1},\dots ,l_m$ and $i_{m-k+1},\dots ,i_m$ are all distinct elements of $\tuple{l}{k}$. All other entries are zero. \\

\ed
 \begin{lemma}\label{deg}
 For all $v_i\in V$,
\begin{enumerate}
\item $d_{v_i}^{+}= \sum_{i_{2}, i_{3}, \dots,  i_{m}=1}^{n} a_{ii_{2}i_{3}\dots i_{m}}.$
\item $d_{v_i}^{-}=\sum_{e\equiv\{v_i,v_{i_2},\dots, v_{i_m}\}\in E,\atop {v_i\in H_e}}{\dfrac{m-|T_e|}{|T_e|}}{a_{i_1 i_2\dots i_{m-1}i}}$\\
\end{enumerate}
 \end{lemma}
 
 \begin{proof}
 \begin{enumerate}
 \item It follows from the definition of out-degree of a vertex.
 \item For an given edge $e$, if $v_{i}\in H_e$, $$\sum _{i_1,\dots, i_{m-1}\atop {v_{i_1},v_{i_2},\dots ,v_{i_{m-1}}\in e\setminus v_i}}{\dfrac{m-|T_e|}{|T_e|}}{a_{i_1 i_2\dots i_{m-1}i}}=1.$$Hence the proof follows.\\
 \end{enumerate}
 \end{proof}
Let $H=(V,E)$ be a  directed hypergraph on $n$ vertices. Let $e$ be an edge and $h\subset H_e$, such that,        $|h|=|H_e|-1$. Now we define $\mathcal{V}_{e}^{\{h\}},\mathcal{V}^{\tuple{l}{k}}\in \mathbb{R}^n$ as follows,$$\mathcal{V}_e^{\{h\}}=\sum\limits_{j,v_j\in h\cup T_e}\alpha_j\bold{1}_j,$$ 
$$\mathcal{V}^{\tuple{l}{k}}=\sum_{j\in\{ \tuple{l}{k}\}}\alpha_j\bold{1}_j,$$ where $\bold{1}_j$'s are the standard basis of $\mathbb{R}^n$, $\alpha_j$'s are scalar and $\tuple{l}{k}$ are distinct elements.
\beg
Let $H=(V,E)$ be a uniform directed hypergraph. $V=\{v_1,v_2,v_3,v_4,v_5,v_6\}$. $e=\{v_1,v_2,v_3,v_4,v_5\}\in E$, such that, $T_e=\{v_1,v_2,v_3\}$. Let us choose, $h=\{v_4\}.$ So, $\mathcal{V}_{e}^{\{h\}}=(\alpha_1,\alpha_2,\alpha_3,\alpha_4,0,0)^T$.
 And, $\mathcal{V}^{l_1,l_2,l_3}=(0,\alpha_5,\alpha_6,\alpha_7,0,0)^T$, where $l_1=2,l_2=3,l_3=4$ and $\alpha_i$'s are  scalars.
\eeg
\begin{theorem}
 Let $H=(V,E)$ be a directed $m$-uniform hypergraph on $n$ vertices such that $m\geq 3$. Then $\mathcal{V}_{e}^{\{h\}}$ and $\mathcal{V}^{\tuple{l}{k}}$ are $H$-eigenvectors of ${A_H}^{+}$ with the eigenvalue zero, where e is any edge in $H$ and $\{\tuple{l}{k}\}$ is $k(1\leq k\leq m-2)$ element subset of $\{1,2,\dots,n\}$. Moreover, their corresponding unit vectors are $Z$-eigenvectors of $A_H^+$ with the $Z$-eigenvalue zero.
 \end{theorem}
 \begin{proof}
The proof follows from the definition of $H$-eigenvalue and $Z$-eigenvalue.
 \end{proof}
\begin{proposition}
$\delta^+\leq\rho(A_H^+)\leq\Delta^+$.
\end{proposition}
\begin{proof}
The result follows from Lemma \ref{bound}.
\end{proof} 
The above bounds can be improved and it is shown in the next theorem.
\bd
A directed hypergraph is called $k$ out(in)-regular if all the vertices have the same out(in) degree $k$.
\ed
\begin{proposition}
Let $H=(V,E)$ is $k$ out-regular hypergraph on $n$ vertices. Then $k$ is an eigenvalue of $A_H^+$. 
\end{proposition}
\begin{proof}
Let $\bold{1}\in \mathbb{R}^n$ be a vector with all the entries are one. \\
Then by the definition of $H$-eigenvalue, $\bold{1}$ is an eigenvector of $A_H^+$ and it's corresponding  eigenvalue is $k$.
\end{proof}
\begin{definition}
Let $A,B\in M_{m,n}$. If there exists a diagonal matrix $D\in M_{2,n}$ such that $B=D^{-(m-1)}AD$. Then we call $A$ and $B$ are diagonal similar tensors.
\end{definition}
In \cite{Shao2013}, it is proved that, $A$ and $B$ are co-spectral.\\ 
\begin{theorem}
 Let $H=(V,E)$ be a directed $m$-uniform hypergraph. Let $\delta^+= d_{1}^+\leq d_{2}^+\leq \dots \leq d_{n-1}^+\leq {d}_n^+=\Delta^+$. Then $${\delta^+} ^{\frac{1}{m}}{d_{2}^+}^{1-\frac{1}{m}}\leq\rho(A_H^+)\leq {\Delta^+} ^{\frac{1}{m}}{d_{n-1}^+}^{1-\frac{1}{m}}.$$
 \begin{proof}
  Case-1 .
  If ${\delta^+}={d_2^+}$ then by the Lemma \ref{bound}, we have $$\rho(A)\geq \min_{1\leq i \leq n}r_{i}(A_H^+)={\delta^+}= {\delta^+} ^{\frac{1}{m}}{d_2^+}^{1-\frac{1}{m}}$$
  
  Case-2.
  Let ${\delta^+}<{d_2^+}$. Take a diagonal matrix, $P=diag(x,1,\dots ,1)$ with $x<1$. Then we have $$r_{1}(P^{-(m-1)}{A_H^+}P)=
  \sum _{i_2,i_3,\dots ,i_m=1}^{n} {(P^{-(m-1)}{A_H^+}P)}_{1i_{2}\dots i_m}=\frac{d_1}{x^{m-1}}$$
 Let $d_{i\{1\}}^+=\abs{ \lbrace e\in E\text{ }:i\in T_e,1\in e\rbrace}$. Then for $2\leq i\leq n$, we have, 
  $$r_{i}(P^{-(m-1)}{A_H^+}P)=
  \sum _{i_2,i_3,\dots ,i_m=1}^{n} {(P^{-(m-1)}{A_H^+}P)}_{ii_{2}\dots i_m}=xd_{i\{1\}}^++d_{i}^+-d_{i\{1\}}^+\geq xd_{i}^+\geq x{d_2^+}$$
We take $x={(\frac{{\delta^+}}{{d_2^+}})}^{\frac{1}{m}}$. Then for each $1\leq i \leq n$, we get ${\delta^+} ^{\frac{1}{m}}{d_2^+}^{1-\frac{1}{m}} \leq r_{i}(P^{-(m-1)}{A_H^+}P)$.
  \\Now $min_{1\leq i\leq n}r_{i}(P^{-(k-1)}{A_H^+}P)\geq {\delta^+} ^{\frac{1}{m}}{d_2^+}^{1-\frac{1}{m}}$.
  Again, $\rho(A_H^+)=\rho(P^{-(m-1)}{A_H^+}P)\geq{\delta^+} ^{\frac{1}{m}}{d_2^+}^{1-\frac{1}{m}}$.
  Similarly, by considering the degree sequence ${\Delta^+}\geq {d_{n-1}^+}\geq \dots \geq {\delta^+}$ and by taking the diagonal matrix $P=diag(x,1,\dots,1)$ with $x={(\frac{{\Delta^+}}{{d_{n-1}^+}})}^{\frac{1}{m}}$
  ($x\geq 1$) we get $\rho({A_H^+})\leq {\Delta^+} ^{\frac{1}{m}}{d_{n-1}^+}^{1-\frac{1}{m}}$. Thus we have $${\delta^+} ^{\frac{1}{m}}{d_2^+}^{1-\frac{1}{m}}\leq\rho(A_H^+)\leq {\Delta^+} ^{\frac{1}{m}}{d_{n-1}^+}^{1-\frac{1}{m}}.$$
 \end{proof}
\end{theorem}

Now for a directed hypergraph $H=(V,E)$, we define a tensor $B=b_{i_1i_2\dots i_m}\in M_{m,n}$ as, 
\[
b_{i_1i_2\dots i_m}=
\begin{cases}
\frac{\abs{T_e}}{m!} \text{ if }  e\equiv\{v_{i_1},v_{i_2},\dots, v_{i_m}\}\in E \\
0 \text{ otherwise }
\end{cases}
\]
\begin{proposition}\label{sym}
$\overline{A_H^+}\equiv B$
\end{proposition} 
\begin{proof}
The proof follows from the definition of $\overline{A_H^+}$.
\end{proof}

\begin{theorem}
Let $H$ be a directed $m$-uniform hypergraph on $n$ vertices, such that, $m$ is even. Then the maximum $Z$-eigenvalue of ${A_H}^{+}$ is less than the the maximum $Z$-eigenvalue of $\overline{{A_H}^{+}}$. 
\end{theorem}
\begin{proof}
Let $\lambda$ and $\mu$ be the maximum $Z$-eigenvalue of ${A_H}^{+}$ and $\overline{A_H^+}$ respectively.
Consider the optimization problem $(P)$, $\max\limits_{x\in \mathbb{R}^n}{A_H}^{+}x^m$ such that $\abs{\abs{x}}_2^m=1$.\\
By Proposition \ref{sym}, ${A_H}^{+}x^m=\overline{A_H^+}x^m$. From the theorem \ref{maxz}, we get that  the solution of $P$ is the maximum $Z$-eigenvalue of $\overline{A_H^+}$, that is, $\mu$. Since any $Z$-eigenvalue is not greater than the optimal value of $P$, then,
 $$\lambda\leq \mu.$$
\end{proof}

\begin{theorem}
$\lambda$ be an $H^+$ eigenvalue of ${A_H^+}$, then $\lambda\geq 0.$
\end{theorem}
\begin{proof}
 $\overline{A_H^+}$ is copositive tensor, since all the entries of $\overline{A_H^+}$ are non-negative. Thus $A_H^+$ is also copositive. Since $0$ is an $H^{+}$ eigenvalue of $A_H^+$, using the Lemma \ref{cop} we get our result.
\end{proof}
Let $H_D(V,E_D)$ be the underlying undirected hypergraph of the direccted hypergraph $H(V,E)$. Then $\forall e\in E$, $T_e\cup H_e$ is an edge in $H_D$ and $\abs{E_D}=\abs{E}$.
\begin{theorem}\label{isspec}
Let $H=(V,E)$ be a directed $m$-uniform hypergraph on $n$ vertices. Let $A_H=A_H^+ +A_H^-$ and $A_{H_D}$ be the adjacency tensor of $H_D$. Then $A_H$ and $A_{H_D}$ are isospectral.
\end{theorem}
\begin{proof}

The $i$-th component of the vector $A_Hx$ is,
\begin{align*}
(A_Hx)_i &=(A_H^+x)_i+(A_H^-x)_i\\
&=\sum\limits_{e\equiv\{v_i,v_{i_2},\dots, v_{i_m}\}\in E,\atop {\text{such that }v_i\in T_e}}x_{i_2}\dots x_{i_m}+\sum\limits_{e\equiv\{v_i,v_{i_2},\dots, v_{i_m}\}\in E,\atop {\text{such that }v_i\in H_e}}x_{i_2}\dots x_{i_m}\\
&=\sum\limits_{{\{v_i,v_{i_2}\dots v_{i_m}\}\in E_D}}x_{i_2}\dots x_{i_m}\\
&=(A_{H_D}x)_i\\
\end{align*}
Hence the result follows.
\end{proof}
\begin{corollary}
Let $H_1=(V,E_1),H_2=(V,E_2)$ be two directed $m$-uniform hypergraph on $n$ vertices such that $H_1$ and $H_2$ have the same underlying undirected hypergraph. Then $A_{H_1}$ and $A_{H_2}$ are isospectral.
\end{corollary}
\begin{lemma}\label{maz}
Let $H=(V,E)$ be a directed $m$-uniform  hypergraph, such that $m$ is even. Then $$\lambda_{\text{max}}^Z(A_H)=max _{||x||=1} x^{t}{A_H}x^{m-1}.$$
\end{lemma}
\begin{proof}
The proof follows from the Theorem \ref{maxz} and the Theorem \ref{isspec}.
\end{proof}
\begin{theorem}

Let $H=(V,E)$ be a directed $m$-uniform hypergraph. Let $H_1, H_2,\dots ,H_p$ be the directed spanning subgraphs of $H$, such that, $E(H_i)$ is the $p$-partitions of $E$.
Then $\lambda_{max}^Z(A_H)\leq \sum\limits_{i=1}^p \lambda_{max}^Z(A_{H_i})$.
\end{theorem}
\begin{proof}
\begin{align*}
\lambda_{max}^Z(H)&=max _{||x||=1} x^{t}{A_H}x^{m-1}\text{ (By the Lemma \ref{maz})}\\
&=max_{||x||=1} x^{t}{\sum\limits_{i=1}^p(A_{H_i})}x^{m-1}\\
&\leq {\sum\limits_{i=1}^p max _{||x||=1} x^{t}(A_{H_i})}x^{m-1}\\
&= \sum\limits_{i=1}^p \lambda_{max}^Z(A_{H_i})\\
\end{align*}
\end{proof}

\section{Directed non-uniform hypergraphs}
\begin{definition}
Let $H=(V, E)$ be a directed non-uniform  hypergraph where $V= \{ v_{1},  v_{2},  \dots,  v_{n} \}$ and $E = \{ e_{1},  e_{2},  \dots,  e_{t} \}$. Let $m=max\{|T_{e_i}\cup H_{e_i}| : e_{i}\in E\}$ be the maximum cardinality of edges,  $m.c.e(H)$ in $H$,  that is rank$(H)=m$.
  Now we define the out-adjacency hypermatrix ${A}_H^+ = (a_{i_{1}i_{2}\dots i_{m}}^+)\in M_{m,n}$  of $H$ as follows.\\
  For any edges $e=\{ v_{l_{1}},  v_{l_{2}},  \dots, v_{l_{k}},v_{l_{k+1}},\dots, v_{l_s}\}\in E$, such that, $T_e=\{ v_{l_{1}},  v_{l_{2}},  \dots, v_{l_{k}}\}$ and $H_e=\{v_{l_{k+1}},\dots, v_{l_s}\}$, 
   
  $$a_{ p_{1}p_{2}\dots p_{t}\dots p_m}^+=\frac{k}{\alpha} \text{ where } k\leq t\leq m-s+k \text{ and }$$ $$\alpha =\sum\limits_{r=0}^{m-s}\left(\sum\limits_{t_{1}, t_{2}, \dots, t_{k} \geq 1, \atop{ \sum t_{i}=r+k}}  \frac{(r+k)!}{t_{1}! t_{2}!\dots t_{k}!}\right)\left(\sum\limits_{t_{1}, t_{2}, \dots, t_{s-k} \geq 1, \atop{ \sum t_{i}=m-k-r}}  \frac{(m-k-r)!}{t_{1}! t_{2}!\dots t_{s-k}!}\right), $$
    where $p_{1}, p_{2},  \dots,  p_{t}$ are chosen in all possible way  from $\{l_{1}, l_{2}, \dots, l_{k}\}$ with at least once for each element of the set and $p_{t+1},\dots,p_m$ are chosen from in all possible way from $\{l_{k+1},\dots , l_{s}\}$ with atleast once for each element of the set. The rest of the entries of $A_H^+$ are zero.\\\\
  Similarly, we define the in-adjacency matrix $A_H^-=(a_{i_{1}i_{2}\dots i_{m}}^-)\in M_{m,n}$ of $H$. For all edges $e=\{ v_{l_{1}},  v_{l_{2}},  \dots, v_{l_{k}},v_{l_{k+1}},\dots, v_{l_s}\}\in E$, such that, $T_e=\{ v_{l_{1}},  v_{l_{2}},  \dots, v_{l_{k}}\}$ and $H_e=\{v_{l_{k+1}},\dots, v_{l_s}\}$ ,  
  $$a_{ p_{1}p_{2}\dots p_{t}\dots p_m}^-=\frac{s-k}{\alpha},  \text{ where } k\leq t\leq m-s+k.$$ Here,  $p_{1}, p_{2},  \dots,  p_{t}$  are chosen in all possible way  from $\{l_{k+1}, \dots, l_{s}\}$ with at least once for each element of the set and $p_{t+1},\dots,p_m$ are chosen from in all possible way from $\{\tuple{l}{k}\}$ with atleast once for each element of the set. The rest of the entries of $A_H^-$ are zero.
\end{definition}
\noindent Clearly, $d_{v_i}^{+}= \sum\limits_{i_{2}, i_{3}, \dots,  i_{m}=1}^{n} a_{ii_{2}i_{3}\dots i_{m}}^+$ and $d_{v_i}^{-}= \sum\limits_{i_{2}, i_{3}, \dots,  i_{m}=1}^{n} a_{ii_{2}i_{3}\dots i_{m}}^-.$\\
\beg
Let $H=(V,E)$ be a  directed non-uniform hypergraph, such that, $V=\{1,2,3,4,5\}, E=\{e_1,e_2\}$, where  $T_{e_1}=\{1,2\}, H_{e_1}=\{3\},T_{e_2}=\{1,4\},H_{e_2}=\{2,5\}$. Then the non zero entries of $A_H^+$ are
 $a_{1233}^+=a_{1223}^+=a_{1123}^+=a_{1213}^+=a_{2133}^+=a_{2123}^+=a_{2213}^+=a_{2113}^+=\frac{1}{4}\text{, }a_{1425}^+=a_{1452}^+=a_{4125}^+=a_{4152}^+=\frac{1}{2}.$\\
 The non zero entries of $A_H^-$ are $a_{3122}^-=a_{3121}^-=a_{3112}^-=a_{3211}^-=a_{3212}^-=a_{3221}^-=a_{3312}^-=a_{3321}^-=\frac{1}{8},\text{ } a_{2541}^-=a_{2514}^-=a_{5241}^-=a_{5214}^-=\frac{1}{2}.$
\eeg
Now, the following theorems for directed non-uniform hypergraphs can be constructed similar to the theorems for directed uniform hypergraphs.
\begin{theorem}
Let $H=(V,E)$ be a directed non-uniform hypergraph with $m.c.e(H)\geq 3$. Suppose $c=\min\limits_{e\in E}\{|T_e|+|H_e|\}$, that is, the $corank(H)$ is $c$ and $\{\tuple{i}{k}\}$, $1\leq k\leq c-2$, is a $k$ element subset of $\{1,2,\dots,n\}$. Then for all $e\in E$, $\mathcal{V}_{e}^{\{h\}}$ and $\mathcal{V}^{\tuple{l}{k}}$ are $H$-eigenvectors of ${A_H}^{+}$ with the $H$-eigenvalue zero. Moreover, their corresponding unit vectors are the $Z$-eigenvectors of $A_H^+$ with the $Z$-eigenvalue zero.
\end{theorem}
\begin{theorem}
$\delta^+\leq\rho(A_H^+)\leq\Delta^+$.\\
\end{theorem}
\begin{theorem}
Let $H=(V,E)$ is $k$ out-regular hypergraph. Then $k$ is an eigenvalue. 
\end{theorem}
\begin{theorem}
 let $H=(V,E)$ be a directed non-uniform hypergraph, such that, $m.c.e(H)=m$. Let $A_H^+$ be the in-adjacency hypermatix of $H$.
 Let $\delta^+\leq d_{2}^+\leq \dots \leq d_{n-1}^+\leq\Delta^+$. Then $${\delta^+} ^{\frac{1}{m}}{d_{2}^+}^{1-\frac{1}{m}}\leq\rho(A_H^+)\leq {\Delta ^+}{\frac{1}{m}}{d_{n-1}^+}^{1-\frac{1}{m}}.$$
\end{theorem}
The above theorems are also hold for the out-adjacency hypermatrix.
\section{Laplacian hypermatrix}
\begin{definition}
  Let $H=(V, E)$ be a directed non-uniform hypergraph on $n$ vertices and $m.c.e(H)=m$. We define the out-Laplacian hypermatrix $L_{H}^+=(l_{i_1\dots i_m}^+)\in M_{m,n}$ for $H$, as $$L_{H}^+=D_{H}^+-{A}_{H}^+,$$ where $D_{H}^+=(d_{i_1\dots i_m}^+)\in M_{m,n},$ is the diagonal hypermatrix, called out-degree hypermatrix of $H$ with $d_{ii\dots i}^+=d^+(v_i)$ and other entries are zero.
The out-signless Laplacian hypermatrix of $H$ is defined as $\mathbb{L}_{H}^+={D_H}^{+} +{A}_{H}^+.$\\

Similarly, we can define in-Laplacian hypermatrix for $H$ as $L_{H}^-=D_{H}^--\mathcal{A}_{H}^-$, where   $D_H^-=(d_{i_1\dots i_m}^+)\in M_{m,n}$ is in-degree hypermatrix of $H$, where the diagonal entry $d_{ii\dots i}^-=d^-(v_i)$ and the rest of the entries are zero. The in-signless Laplacian hypermatrix of $H$ is defined as $\mathbb{L}_{H}^-={D_H}^{-} +\mathcal{A}_{H}^-.$
\end{definition}

We denote $L_H=L_H^++L_H^-$ and $\mathbb{L}_H=\mathbb{L}_H^++\mathbb{L}_H^-$ 
\begin{theorem}
Let $H=(V, E)$ be a directed $m$-uniform hypergraph. Let $L_{H_D}$ and $\mathbb{L}_{H_D}$ be the Laplacian and signless Laplacian hypermatrices, respectively of $H_D$.
Then $L_H$ and $L_{H_D}$ are isospectral, as well as, $\mathbb{L_H}$ and $\mathbb{L}_{H_D}$ are isospectral.
\end{theorem}
\begin{proof}
Let $D_{H_D}$ be the degree hypermatrix of $H_D$ i.e, $D_{H_D}$  is a diagonal hypermatrix where diagonal entries are the degrees of the $H$ vertices of $H_D$. Then,$$D_{H_D}=D_H^++D_H^{-}.$$ Now by the Theorem \ref{isspec} we get the desired result.
\end{proof}
\begin{corollary}
$L_H$ is copositive for directed uniform hypergraph.  
\end{corollary}

\begin{theorem}\label{lapp}
Let $H=(V, E)$ be a directed hypergraph and $L_{H}^+$ and be its corresponding out-Laplacian hypermatrix. Then,
\begin{enumerate}[(i)]
\item $0$ is an $H$- and $Z$- eigenvalue of $L_{H}^+$.
\item $\rho(L_{H}^+)\leq 2\Delta^+$.
\item If $\lambda$ be an $H$-eigenvalue of $L_H^+$, then $0\leq\lambda\leq 2\Delta^+$.
\item If $m.c.e(H)\geq 3$, $(d^{+}(i),\bold{1}_i)$ is an eigenpair of $L_H^+$.\\
\item If $m.c.e(H)\geq 3$, then, $0\leq \lambda_{min}(L_{H}^+)\leq \delta^+\leq \Delta^+ \leq \lambda_{max}(L_H^+) \leq 2\Delta^+$, where $\lambda_{min}$ and $\lambda_{max}$ are the minimum and maximum eigenvaluse of ${L}_H^+$, respectively. 
\end{enumerate}
\end{theorem}
\begin{proof}
\begin{enumerate}[(i)]
\item  Choose a vector $x=\bold{1}$.\\
Now, $(L_H^+x)_i=0, 1\leq i\leq n$. Hence our result follows.
\item Using the Theorem \ref{disk}, we have $$|\lambda -l_{ii\dots i}^+|\leq  \sum\limits_{{i_{2}, i_{3}, \dots,  i_{m}=1,\atop{\delta_{ii_2\dots i_m=0}}}}^{n} |l_{ii_{2}i_{3}\dots i_{m}}^+|\leq \Delta^+,$$
Thus,  $\rho(L_{H}^+)\leq 2\Delta^+$.
\item 
From the Theorem \ref{disk}, for a given eigenvalue $\lambda$ there exists an $i$, such that,
\begin{align*}
\mid \lambda -d^+(i)\mid &\leq d^+(i)\text{, } 1\leq i\leq n.
\end{align*}
Thus, $\lambda$, $0\leq\lambda\leq 2\Delta^+$.

\item Let us choose a $k$, such that, $1\leq k\leq n$. Now we have the following cases,
Case-1(When $i=k$):
\begin{align*}
(L_H^+x)_i&=\sum\limits_{i_2,i_3,\dots, i_m=1}^n l_{ki_2\dots i_m}^+x_{i_2}x_{i_3}\dots x_{i_m}\\
&=l_{kk\dots k}^+\\
&=d^+(k).x_i
\end{align*}
Case-2 (When $i\neq k$):
\begin{align*}
(L_H^+x)_i&=\sum\limits_{i_2,i_3,\dots, i_m=1}^n l_{ii_2\dots i_m}^+x_{i_2}x_{i_3}\dots x_{i_m}\\
&=0\\
&=d^+(k).x_i
\end{align*}

Hence the result follows.
\item Our result follows from the part (iii) and (iv) of this theorem.
\end{enumerate}
\end{proof}

\begin{theorem}
Let $H=(V, E)$ be a directed hypergraph and $\mathbb{L}_{H}^+=\mathit{l_{i_1\dots i_m}^+}$ be the corresponding signless Laplacian hypermatrix. Then,
\begin{enumerate}[(i)]
\item If $m.c.e(H)\geq 3$, $(d^{+}(i),\bold{1}_{i})$ is an eigenpair.
\item $2\delta^+\leq\rho(\mathbb{L}_H^+)\leq 2\Delta^+$.
\item If $m.c.e(H)\geq 3$, $0\leq \lambda_{min}(\mathbb{L}_{H}^+)\leq \delta^+\leq \Delta^+ \leq \lambda_{max}(\mathbb{L}_H^+) \leq 2\Delta^+$, where $\lambda_{min}$ and $\lambda_{max}$ are the minimum and maximum eigenvaluse of $\mathbb{L}_H^+$, respectively. 
\item Suppose $H=(V,E)$ is directed $m$-uniform hypergraph, such that, $m$ is even. Suppose $E=\{e_1,e_2,\dots,e_t\}$. If $\abs{ \cap_{i=1}^t  e_i}\geq 1$ then $0$ is an eigenvalue of $\mathbb{L}_{H}^+$.
\item If $H$ is a $k$-out regular directed hypergraph then $2k$ is an eigenvalue of $\mathbb{L}_H^+$.
\end{enumerate}
\end{theorem}

\begin{proof}
\begin{enumerate}[(i)]
\item The proof is similar to the proof of part (iv) of the Theorem \ref{lapp} .
\item Using the Lemma \ref{bound}, we get our desired result.
\item 
Let $\lambda$ be an eigenvalue,
From the Theorem \ref{disk}, for a given $\lambda$ there exists an $i$, where, $1\leq i\leq n$, such that,  
$\abs{ \lambda -d^+(i)}\leq d^+(i)$.\\ 
So, $0\leq \lambda \leq 2d^+(i)$. Now using (i) of this theorem we get our desired result.\\
\item As $\mid \cap_{i=1}^t  e_i\mid\geq 1$, without loss of generality, suppose $v_1\in \cap_{i=1}^t  e_i$.
 Let us choose $x=(x_i)\in M_{1,n}$, such that,\\
$x_i= 
 \begin{cases}
 1, \text{ }i=1,\\
 -1, \text{ }i\neq1.\\
 \end{cases}$
 \\
Case-1(When $i=1$):
\begin{align*}
(\mathbb{L}_H^+x)_1&=\sum\limits_{i_2,i_3,\dots, i_m=1}^n \mathit{l}_{1i_2\dots i_m}^+x_{i_2}x_{i_3}\dots x_{i_m}\\
&=\mathit{l}_{11\dots 1}^++(-1)^{m-1}\sum\limits_{i_2,i_3,\dots, i_m=1}^n \mathit{l}_{1i_2\dots i_m}^+\\
&=0.\\
\end{align*}
Case-2 (When $i\neq 1$):
\begin{align*}
(\mathbb{L}_H^+x)_i&=\sum\limits_{i_2,i_3,\dots, i_m=1}^n\mathit{l}_{ii_2\dots i_m}^+x_{i_2}x_{i_3}\dots x_{i_m}\\
&=-\mathit{l}_{ii\dots i}^++(-1)^{m-2}\sum\limits_{i_2,i_3,\dots, i_m=1}^n \mathit{l}_{ii_2\dots i_m}^+\\
&=0.
\end{align*}
Thus our result follows.
\item Choose a vector $x=\bold{1}$.\\
Then  $2d$ becomes an eigenvalue of $\mathbb{L}_H^+$ with the eigenvector $x$.\\
\end{enumerate}
\end{proof}
The above two theorems also hold for $L_H^-$ and $\mathbb{L}_H^-$.
\section{Connectivity in a directed hypergraph}
As in graph, we have the following definition for a directed hypergraph to be strongly connected.
\begin{definition} 
A directed walk from a vertex $v_0$ to $v_k$ on a directed hypergraph $H = (V, E)$ is an alternating sequence of vertices and (directed) edges, $\{v_0, e_1, v_1, e_2, v_2,\dots v_{k-1},e_{k}, v_k \}$, where the vertex $v_i\in T_{e_{i+1}},$ $\forall i\in \{0,1\dots,k-1\}$ and $v_i\in H_{e_{i}},$ $\forall i\in \{1,2,\dots,k\}$.
A directed hypergraph is strongly connected if for every ordered pair of vertices $(v_i,v_j)$, there exists a directed walk from $v_i$ to $v_j$.
\end{definition}

Take a directed hypergraph $H = (V,E)$, where $V = \{v_1,v_2, v_3, v_4\}$, $E = \{e_1, e_2\}$, $T_{e_1} = \{v_1, v_3\}, H_{e_1} = \{v_2\}$, $T_{e_2} = \{v_4, v_2\}, H_{e_2} = \{v_3\}$. The directed hypergraph is not strongly connected as there is no directed walk from the vertex $v_2$ to $v_1$, but, the corresponding adjacency hypermatrix is weakly irreducible. An underlying weakly irreducible  hypermatrix does not reflect  the strongly connectedness in the hypergraph.
Thus, to capture the strongly
connectivity information of a directed hypergraph, we are introducing the concept of weak* irreducible hypermatrix.
\begin{definition}  
Let $A = a_{i_1i_2\dots i_m}\in M_{m,n}$ be a real hypermatrix. Construct a directed graph $G^{*} = (V,E)$, with the vertex set $V=[n]$, and directed edges $\vec{ij}\in E$, if $$\sum\limits_{i_2,\dots,i_{m-1}=1}^n |a_{ii_2\dots i_{m-1}j}|>0.$$ The hypermatrix $A$ is called weak* irreducible if $G^{*}$ is strongly connected.
\end{definition}
Clearly a weak* irreducible hypermatrix is always weakly irreducible hypermatrix. It is easy to conclude that a directed hypergraph is strongly connected if and only if it's adjacency hypermatrix is weak* irreducible hypermatrix. 
\section{Discussion}
Many results on spectral properties  of directed and undirected hypergraphs which have been developed in earlier research work also follow for the same of  generalized uniform directed hypergraph. We did not mention all of them in this article as they are straight forward application of our definition. It is evident that the Perron-Frobenius theorem also hold for the weak* irreducible hypermatrix. 

\section*{Acknowledgment}
AB sincerely acknowledge the financial support from the grant "MATRICS" of the Science and Engineering Research Board (SERB), India (reference no.MTR/2017/000988).

\end{document}